%% file: galerkinME_revision2_arxiv.tex
\documentclass{siamart171218}
\usepackage{amsfonts,amsmath,color,amssymb}
\usepackage{graphicx}
\input{generic}

\usepackage{pgfplotstable}
\pgfplotsset{compat=1.13}

\pgfplotsset{
  /pgfplots/error bar legend 1/.style={
    legend image code/.code={
        \draw[sharp plot,mark=-,mark repeat=2,mark phase=1,color=red,dotted,##1]
        plot coordinates { (0.3cm, -0.15cm) (0.3cm,0cm) (0.3cm, 0.15cm) };%
        \draw[mark repeat=2,mark phase=2,##1]
        plot coordinates {(0cm,0cm) (0.3cm,0cm) (0.6cm,0cm)};%
        }}}

        \pgfplotsset{
  /pgfplots/error bar legend 2/.style={
    legend image code/.code={
        \draw[sharp plot,mark=-,mark repeat=2,mark phase=1,color=blue,solid,##1]
        plot coordinates { (0.3cm, -0.15cm) (0.3cm,0cm) (0.3cm, 0.15cm) };%
        \draw[mark repeat=2,mark phase=2,##1]
        plot coordinates {(0cm,0cm) (0.3cm,0cm) (0.6cm,0cm)};%
        }}}
        \pgfplotsset{
  /pgfplots/error bar legend 3/.style={
    legend image code/.code={
        \draw[sharp plot,mark=-,mark repeat=2,mark phase=1,color=black,dashed,##1]
        plot coordinates { (0.3cm, -0.15cm) (0.3cm,0cm) (0.3cm, 0.15cm) };%
        \draw[mark repeat=2,mark phase=2,##1]
        plot coordinates {(0cm,0cm) (0.3cm,0cm) (0.6cm,0cm)};%
        }}}

        \newcommand{\RR}{{\mathbb{R}}}

\DeclareMathOperator*{\argmin}{argmin}
\newtheorem{remark}[theorem]{Remark}

\newtheorem{num_example}[theorem]{Example}

\begin{document}
\bibliographystyle{siam}
\pagestyle{myheadings}
\markboth{D. Palitta and V.~Simoncini}{Optimality properties of algebraic (Petrov-)Galerkin methods}

\title{Optimality properties of Galerkin and
Petrov-Galerkin methods for linear matrix equations%
\thanks{Version of November 14, 2019}}
\author{Davide Palitta\thanks{Research Group Computational Methods in Systems
and Control Theory (CSC),
Max Planck Institute for Dynamics of Complex Technical Systems, Sandtorstra\ss{e} 1, 39106 Magdeburg, Germany. \texttt{palitta@mpi-magdeburg.mpg.de}} \and Valeria Simoncini\thanks{Dipartimento di Matematica,
Alma Mater Studiorum Universit\`a di Bologna,
Piazza di Porta San Donato  5, I-40127 Bologna, Italy,
 and IMATI-CNR, Pavia, Italy. {\tt valeria.simoncini@unibo.it}}}
\maketitle

\begin{center}{\it Dedicated to Volker Mehrmann on the occasion of his 65th birthday}\end{center}

\begin{abstract}
Galerkin and Petrov-Galerkin methods are some of the most successful solution procedures in numerical analysis. 
Their popularity is mainly due to the optimality properties of their approximate solution.
We show that these features carry over to the (Petrov-)Galerkin methods applied for
 the solution of linear matrix equations.

Some novel considerations about the use of Galerkin and Petrov-Galerkin schemes in the numerical treatment of general linear matrix equations are expounded and the use of constrained minimization techniques in 
the Petrov-Galerkin framework is proposed. 

\end{abstract}



\section{Introduction}

Many state-of-the-art solution procedures for 
 algebraic linear systems of the form
\begin{equation}\label{eqn:mainkron}
{\cal M} x = f,
\end{equation}
where ${\cal M}\in\RR^{N\times N}$ and $f\in\RR^N$, 
are based on projection. Given a subspace
${\cal K}_m$ of dimension $m$, and a matrix ${\cal V}_m$ whose orthonormal columns span
${\cal K}_m$, these methods seek an approximate solution $x_m = {\cal V}_m y_m$ for some
$y_m\in\RR^m$ by imposing certain conditions. The most successful projection procedures impose either a 
{\it Galerkin} or a {\it Petrov-Galerkin} condition on the residual $r_m =  f - {\cal M} x_m$. See, e.g., \cite{Saad2003}.
These conditions are very general, and they are at the basis of many approximation methods, beyond
the algebraic context of interest here; any approximation strategy associated with an inner product
can determine the projected solution by one of such a condition. Finite element methods,
both at the continuous and discrete levels, strongly 
rely on this methodology; see, e.g., \cite{Strang.Fix.73}, but also eigenvalue problems \cite{Saad1992}.

It is very important to realize that this is a methodology, not a single method: the approximation
space can be generated independently of the condition, and in a way to make the computation
of $y_m$ more effective, while obtaining a sufficiently accurate approximation with the smallest possible
space dimension.

A fundamental property of the Galerkin methodology is obtained whenever the coefficient matrix 
${\cal M}$ is
symmetric and positive definite (spd): the Galerkin condition  on the residual corresponds to
minimizing the error vector in the norm associated with ${\cal M}$ over the approximation space. This property is
at the basis of the convergence analysis of methods such as the Conjugate Gradient (CG) \cite{Hestenes1952},
and it ensures monotonic convergence, in addition to finite termination, in exact
precision arithmetic.

When $\mathcal{M}$ is not spd, the application of the Galerkin method does not automatically imply a minimization of the error norm. Nevertheless, a certain family of Petrov-Galerkin procedures still fulfills an optimality property. Indeed, these methods minimize the residual norm over the space $\mathcal{MK}_m$. See, e.g., \cite{Saad2003}. Some of the most popular solvers for linear systems such as MINRES \cite{Paige1975} and GMRES \cite{Schultz1986} belong to this collection of methods.

In the past decades, projection techniques have been successfully used to solve linear matrix 
equations of the form
\begin{equation} \label{eqn:matrixeq}
A_1 X B_1 + A_2 X B_2 + \ldots + A_\ell X B_\ell = F,
\end{equation}
that have arisen as a natural algebraic model for discretized partial differential equations (PDEs),
possibly including stochastic terms or parameter dependent coefficient matrices 
\cite{Baumann2018,Benner.Damm.11,Powell.Silvester.Simoncini.17,Palitta2016}, for PDE-constrained optimization problems \cite{Stoll2015}, data assimilation \cite{Freitag2018}, and many other applied contexts, including
building blocks of other numerical procedures \cite{Lin2013}; see also \cite{Simoncini2014,Benner2013} for further
references.

The general matrix equation (\ref{eqn:matrixeq}) covers two well known cases, the (generalized)
Sylvester equation (for $\ell = 2$), and the Lyapunov equation
\begin{eqnarray}\label{eqn:Lyap}
A X  +  X A^T  = F,
\end{eqnarray}
which plays a crucial role in many applications such as control and system theory \cite{Benner2005,Antoulas.05}, and in the solution of Riccati equations by the Newton method, in which
 a Lyapunov equation needs to be solved at each Newton step. See, e.g., \cite{Mehrmann1991}.

The aim of this paper is to generalize the optimality properties of the Galerkin and Petrov-Galerkin methods to matrix equations,
and to extend other convergence properties of CG and some related schemes to the matrix setting.
Some of the proposed results are new, some others can be found in articles scattered in the
literature in different contexts. We thus provide a more uniform presentation of these results.

To introduce a matrix version of the error and residual minimization, we first
recall the relation between matrix-matrix operations and
Kronecker products. Indeed, if $\otimes$ denotes the Kronecker product and ${\cal T} := B^T \otimes A$, then
$$
Y = A X B \quad \Leftrightarrow \quad  y={\cal T} x, \quad x = {\rm vec}(X),\; y={\rm vec}(Y) ,
$$
where the usual ``vec($\cdot$)'' operator stacks the columns of the argument matrix one
after the other into a long vector.

\section{The Galerkin condition}
In this section we first recall the result connecting the Galerkin condition on the
residual with the minimization of the error norm when this is applied to the solution of linear systems,
 and then we show that similar results can be obtained also in the matrix equation setting. For the rest of the section
we assume that ${\cal M}$ in \eqref{eqn:mainkron} is symmetric and positive definite.


\subsection{The linear system setting}
Let $x_m=V_m y_m$ be an approximation to the true solution of (\ref{eqn:mainkron}), and
let $e_m = x - x_m$, $r_m = f - {\cal M} x_m$ be the associated error and residual, respectively.
We recall that imposing the Galerkin condition yields
\begin{eqnarray}\label{eqn:gal_cond}
V_m^T r_m =0 \quad \Leftrightarrow \quad
V_m^T {\cal M} V_m y_m = V_m^T f.
\end{eqnarray}
Note that the coefficient matrix $V_m^T {\cal M} V_m$ is symmetric and positive definite.
Solving this system  yields the ``projected'' vector $y_m$, so as to completely 
define $x_m$.

Let $\|e_m\|_{\cal M}^2 : = e_m^T {\cal M} e_m$ be the ${\cal M}$-norm associated with
the spd matrix ${\cal M}$.
For the error we thus have
\begin{eqnarray}\label{eqn:err_M}
\|e_m\|_{\cal M}^2 = \| {\cal M}^{1/2}(x-x_m)\|^2 = \|{\cal M}^{1/2}x - {\cal M}^{1/2} V_m y_m\|^2.
\end{eqnarray}
The minimization of the error ${\cal M}$-norm thus corresponds to solving the least
squares problem on the right, which gives
$$
({\cal M}^{1/2} V_m)^T {\cal M}^{1/2} V_m y_m = ({\cal M}^{1/2} V_m)^T {\cal M}^{1/2}x,
$$
which, upon simplifications of the transpositions yields $V_m^T {\cal M} V_m y_m = V_m^T f$,
that is, using (\ref{eqn:gal_cond}), $V_m^T r_m =0$.

\subsection{Galerkin method and error minimization for matrix equations}\label{sec:galerkin_ME}
To simplify the presentation, we first discuss Galerkin projection with the Lyapunov equation.
Given the equation~(\ref{eqn:Lyap}) with $A$ spd and $F=F^T$, then it can be shown
that $X$ is symmetric. Letting ${\rm range}(V_k)$  be an approximation space,
we can determine an approximation to $X$ as $X_k = V_k Y_k V_k^T$, which in vector
notation is written as $\text{vec}(X_k) = (V_k\otimes V_k) {\rm vec}(Y_k)$. The matrix $Y_k$ is
obtained by imposing the Galerkin condition in a matrix sense to the
residual matrix $R_k = F-(AX_k+X_kA)$, that is
$$
V_k^T R_k V_k = 0 \quad \Leftrightarrow \quad
(V_k\otimes V_k)^T r_k = 0,
$$
where $r_k = {\rm vec}(R_k)$. Therefore, if one writes the Lyapunov equation 
 by means of the Kronecker formulation, the obtained approximation space is
${\cal K}_m={\rm range}(V_k\otimes V_k)$. 

We explicitly  notice that $X_k$ belongs to ${\rm range}(V_k)$,
which is much smaller than range$(V_k\otimes V_k)$. Therefore, by sticking to the matrix
equation formulation, we expect to build a much smaller approximation space than
if a blind use of the Kronecker form were used. In other words, by exploiting the
original matrix structure, no redundant information is sought after. In section
section~\ref{Comparisons with the Kronecker formulation} we provide a rigorous
analysis of this argument. See also \cite{Kressner.Tobler.10}.
To be able to exploit the derivation in (\ref{eqn:err_M})
we will define an error matrix and the associated inner product.

The generalization to the multiterm linear equation (\ref{eqn:matrixeq}) requires
the definition of two approximation spaces, since the right and left coefficient
matrices are not necessarily the same. Therefore, let range$(V_k)$ and range$(W_k)$ be
two approximation spaces of dimension $k$ each\footnote{In principle, we can have $\text{dim}(\text{range}(V_k))\neq\text{dim}(\text{range}(W_k))$. Here we consider $\text{dim}(\text{range}(V_k))=\text{dim}(\text{range}(W_k))=k$ 
for the sake of simplicity in the presentation.}, and let us 
write the approximation to $X$ as $X_k  = V_k Y_k W_k^T$.
With the residual matrix $R_k = F - \sum_{j=1}^\ell A_j X_k B_j$, the Galerkin condition now takes the form
$$
V_k^T R_k W_k = 0 \quad \Leftrightarrow \quad
(W_k\otimes V_k)^T r_k = 0,
$$
where $r_k = {\rm vec}(R_k)$, so that
${\cal K}_m={\rm range}(W_k\otimes V_k)$ with $m=k^2$ in the Kronecker formulation.

To adapt the error minimization procedure to the matrix equation setting  we first introduce
a matrix norm, that allows us to make a connection with the ${\cal M}$-norm of the error vector.
A corresponding derivation for $\ell=2$ can be found, for instance, in \cite[p. 2557]{Vandere_Vandew.10} and \cite[p. 149]{Benner2014}.

\begin{definition} \label{def:S}
Let 
\begin{equation}\label{defS}
\begin{array}{lrll}
  {\cal S} :& \mathbb{R}^{n\times p}&\rightarrow&\mathbb{R}^{n\times p}\\
  & X &\mapsto& \displaystyle\sum_{j=1}^\ell A_j X B_j,\\
  \end{array} 
\end{equation}
and ${\cal S}_\ell =\sum_{j=1}^\ell B_j^T \otimes A_j$.
We say that the operator ${\cal S}$ is symmetric and positive
definite if for any $0 \ne x\in \RR^{np}$, $x={\rm vec}(X)$, with $X\in\RR^{n\times p}$,
 it holds that ${\cal S}_\ell = {\cal S}_\ell^T$ and
$x^T {\cal S}_\ell x >0$, where 
$$
x^T {\cal S}_\ell x = {\rm trace}\left( \sum_{j=1}^\ell X^T A_j X B_j\right).
$$
The norm induced by this operator will be denoted 
by $\|X\|_{\cal S}$.
\end{definition} 

Note that any linear operator $\mathcal{L}:\mathbb{R}^{n\times p}\rightarrow\mathbb{R}^{n\times p}$ can be written in the form \eqref{defS} with a uniquely defined minimum number of terms $\ell$ called the \emph{Sylvester index}. See \cite{Konstantinov2000}.

Assuming $\mathcal{S}$ to be spd, in the following proposition we show that the error matrix is minimized in the ${\cal S}$-norm.

\begin{proposition}\label{th_errormin}
Let ${\cal S}(X) = F$ with ${\cal S}: X \mapsto \sum_j A_jXB_j$ spd,
and let $\text{range}(V_k)$, $\text{range}(W_k)$ be the constructed
approximation spaces, so that $X_k = V_k Y_k W_k^T$ is the Galerkin approximate
solution. 
Then
$$
\|X-X_k\|_{\cal S} = \min_{Z=V_k Y W_k^T\atop Y\in\RR^{k\times k}} \|X-Z\|_{\cal S} .
$$
\end{proposition}

\begin{proof}
Let $e_k = {\rm vec}(X-X_k)$ be the error vector, $r_k = {\rm vec}(F-\sum_j A_j X_kB_j)$ the residual vector,
${\cal K}_m = \text{range}(W_k \otimes V_k)$  the approximation space and ${\cal S}_\ell =
\sum_j B_j^T \otimes A_j$ the coefficient matrix.  Then, since $\mathcal{S}$ is spd by assumption, also ${\cal S}_\ell$ is spd,
and the Galerkin condition ${\cal V}_m^T r_k = 0$, $\mathcal{V}_m=W_k \otimes V_k$, corresponds to the
minimization of the error. More precisely, it holds
$$
\|e_k\|_{\cal S}^2 = e_k^T {\cal S}_\ell e_k = {\rm trace}((X-X_k)^T {\cal S} (X-X_k)) = \|X-X_k\|_{\cal S}^2,
$$
and the proof is completed.
\end{proof}

Proposition~\ref{th_errormin} states that as long as the approximation
spaces are expanded, the error will decrease monotonically in the considered norm.
A Galerkin approach for a multiterm linear matrix equation was for instance employed
in \cite{Powell.Silvester.Simoncini.17}; the proposition above thus ensures that 
under the stated hypotheses on the data the method will minimize the error as the
approximation spaces grow. See also Example~\ref{Ex.3}.

A result similar to the one stated in Proposition~\ref{th_errormin} can be found in \cite{Kressner.Tobler.10} 
where the authors consider specific approximation spaces and assume $\mathcal{S}$ to be a so-called \emph{Laplace-like}
 operator. Proposition~\ref{th_errormin} shows the strength of the Galerkin method, also 
in the general matrix equation setting. Indeed, the optimality condition of the Galerkin method does neither depend on the adopted approximation spaces nor on the definition of $\mathcal{S}$, as long as this is spd.


Given a general linear matrix equation \eqref{eqn:matrixeq} written in the form $\mathcal{S}(X)=F$, 
one would like to characterize the symmetry and positive definiteness of $\mathcal{S}$ by  
looking only at the properties of the matrices $A_j$ and $B_j$ and avoid the construction of 
the large matrix $\mathcal{S}_\ell$.

 Assuming $\ell$ to be the Sylvester index of $\mathcal{S}$, it is easy to show that $\mathcal{S}$ is a symmetric operator if and only if the matrices $A_j$ and $B_j$ are 
symmetric for all $j=1,\ldots,\ell$, whereas, in general, it is not possible to identify 
the positive definiteness of $\mathcal{S}$ by examining the spectral distributions of $A_j$ and 
$B_j$, even when these are completely known. See, e.g., \cite{Lancaster1970}. Note 
that for $\mathcal{S}$ to be spd it is not necessary for all the 
$A_j$'s and $B_j$'s to be positive definite. 
Nevertheless, if $A_j$, $B_j$ are positive definite for all 
$j=1,\ldots,\ell$, then $\mathcal{S}$ is positive definite; see, e.g., \cite[Proposition 3.1]{Vandere_Vandew.10}
for $\ell =2$.
Therefore, in the case of the Lyapunov equation with $A$ spd, also the operator $\mathcal{S}$ is spd
and it holds that
$$
\|X\|_{\cal S}^2 = 2\, {\rm trace}( X^T A X).
$$
Another case where the properties of $\mathcal{S}$ can be determined in terms of the (symmetric)
coefficient matrices $A_j$ and $B_j$
is the Sylvester operator ${\cal S} : X \mapsto  A X + X B$. 
By exploiting 
the property of the Kronecker product, it holds that $\mathcal{S}$ is positive definite if
and only if $\lambda_i(A)+\lambda_j(B) >0$ for all $i$s and $j$s.
Moreover, the norm $\|\cdot\|_{\cal S}$ can be
written as $\|X\|_{\cal S}^2 = {\rm trace}( X^T A X) + {\rm trace}( X B X^T)$.

\begin{remark}
 Consider the Lyapunov equation 
\eqref{eqn:Lyap} with the spd coefficient matrix $A$,  and let
$E_k := X-X_k$ be the corresponding error matrix. Then, the previous discussion shows that
$$
\|E_k\|_{\cal S}^2 = \min_{Z=V_k Y W_k^T\atop Y\in\RR^{k\times k}} \|X-Z\|_{\cal S}^2
= 2\, {\rm trace}( E_k^T A E_k).
$$
\end{remark}

In the remark above we have not specified whether the known term $F$ in (\ref{eqn:Lyap}) 
needs to be
symmetric. If $F$ is symmetric, then indeed the two spaces can coincide, and $E_k$ is also symmetric. On the other hand,
if  $F$ has the form $F=F_1 F_2^T$, possibly low rank,  natural choices as
approximation spaces are such that ${\rm range}(F_1)\subseteq {\rm range}(V_k)$ and
${\rm range}(F_2)\subseteq {\rm range}(W_k)$, so that the (vector) residual is orthogonal
to $F_2\otimes F_1$. A possible alternative could use $V_k=W_k$ such that 
${\rm range}(F_1), {\rm range}(F_2) \subseteq {\rm range}(V_k)$, where however in
general we expect ${\rm range}(V_k)$ to have larger dimension than in the previous case.

\begin{num_example}\label{Ex.3}
{\rm
By applying the stochastic Galerkin methodology for the discretization of elliptic stochastic PDEs 
\cite{Babuska2004}, the resulting algebraic formulation 
can be written as the linear matrix equation~\eqref{eqn:matrixeq} with typically $\ell >2$.
When dealing with the stochastic steady-state diffusion problem with homogeneous Dirichlet
boundary conditions, 
%
the symmetric matrices $A_j$ and $B_j$ 
may  not all be positive definite; nonetheless, 
 the associated operator $\mathcal{S}$ is symmetric and indeed {\it positive definite} 
(see, e.g., \cite{Powell2009}), so that the previous theory applies. 
In the following we consider the 
Galerkin approach developed in \cite{Powell.Silvester.Simoncini.17} -- based on the rational
Krylov subspace -- to
illustrate the monotonic decrease of the error $\mathcal{S}$-norm as
the approximation space increases\footnote{The Matlab code is available at {\tt http://www.dm.unibo.it/\textasciitilde simoncin/software.html.}}.
%
%
 We generate $A_j$ and $B_j$ as the second test case  
in the S-IFISS package~\cite{Silvester2015}
with the default setting for all the requested parameters.
This yields} a linear matrix equation of the form~\eqref{eqn:matrixeq} with 
$\ell=6$, $A_j\in\mathbb{R}^{n\times n}$, $n=225$, and $B_j\in\mathbb{R}^{p\times p}$, $p=56$. 
The right-hand side $F$ has rank 1.
Thanks to the small problem dimension, we could compute the vectorized solution $x\in\mathbb{R}^{np}$
as $x = {\rm vec}(X)= {\cal S}_\ell^{-1}f$ (Matlab function ``$\setminus$''), to be used as a reference
``exact'' solution.
In particular, if $X_k$ denotes the approximate solution obtained after $k$ iterations of the Galerkin method, we compute
$\|X-X_k\|_{\mathcal{S}}/\|X\|_{\mathcal{S}}$ until this falls below $10^{-6}$.
Figure~\ref{Fig.Ex.3} displays the history of this relative error $\mathcal{S}$-norm, illustrating
the expected monotonically non-increasing curve.

\begin{figure}
  \centering    
  \caption{Example~\ref{Ex.3}. Relative error energy norm. }\label{Fig.Ex.3}
	\begin{tikzpicture}
  \begin{semilogyaxis}[width=0.8\linewidth, height=.27\textheight,
    legend pos = north east,
    xlabel = $k$, ylabel = Relative Error ($\mathcal{S}$-norm)]
    \addplot+[thick] table[x index=0, y index=1]  {errenergynorm_newexample.dat};
  \end{semilogyaxis}          
\end{tikzpicture}
\end{figure}

\end{num_example}

\section{Convergence properties}\label{Convergence properties}

In the previous section we have shown that the Galerkin condition leads to 
a minimization of the error $\mathcal{S}$-norm and this property does 
not depend on the selected space $\mathcal{K}_k=\text{range}(V_k)$. In actual computations, 
a measurable estimate of the error is needed and in \cite{Simoncini.Druskin.09} an upper 
bound on the Euclidean norm of the error is 
provided in the case of the Lyapunov equation \eqref{eqn:Lyap} with rank-one right-hand side 
$F=b b^T$ with $\|b\|=1$, and a positive definite but not necessarily symmetric $A$.
By exploiting the closed-form of the solution $X$,
the authors showed
that
$$
\|X - X_k\|_2 \leq 2 \int_0^\infty e^{-t\alpha_{\min}(A)} \|x-x_m\|_2 dt, \quad \alpha_{\min}(A) = \lambda_{\min}((A+A^T)/2),
$$
where $x =  e^{-t A} b$, $x_k = V_k e^{-t A_k} e_1$, $A_k:=V_k^TAV_k$, and $\|\cdot\|_2$ denotes the Euclidean norm. 

 This led to the following proposition when the selected approximation space is the Krylov subspace $\text{range}(V_k)=K_k(A,b) = {\rm span}\{b, Ab, \ldots, A^{k-1}b\}$ and $A$ is symmetric.

\begin{proposition}[\cite{Simoncini.Druskin.09}]\label{prop:sym} Let $A$ be spd, and
let $\lambda_{\max}$ and $\lambda_{\min}$ be the largest and the smallest eigenvalue of $A$, respectively. Denoting by $\hat\kappa=(\lambda_{\max}+\lambda_{\min})/2\lambda_{\min}$ the condition number of the spd matrix $A+\lambda_{\min} I$,
then the Galerkin approximate solution $X_k=V_k Y_k V_k^T$ satisfies
\begin{eqnarray}
\|X- X_k\|_2 & \leq &
 2\frac {  \sqrt{\hat\kappa}+1 } 
 { \lambda_{\min}\sqrt{\hat\kappa}} 
\left ( \frac{\sqrt{\hat \kappa} -1}{\sqrt{\hat\kappa}+1}\right )^k .
\label{eqn:bound1} 
\end{eqnarray}

\end{proposition}

This bound, in terms of slope as $k$ increases, was shown to be sharp in \cite{Simoncini.Druskin.09}.
Notice that the bound \eqref{eqn:bound1} holds also for the Frobenius norm of the error, namely $\|X-X_k\|_F$. Indeed, we can still write $
\|X - X_k\|_F \leq 2 \int_0^\infty e^{-t\alpha_{\min}(A)} \|x-x_m\|_F dt$ and the rest of the proof of Proposition~\ref{prop:sym} makes use of bounds for norms of vectors only, for which the Euclidean and the 
Frobenius norms coincide. See \cite[Proposition 3.1]{Simoncini.Druskin.09} for more details.
The bound can be generalized to the use of other spaces, such as rational Krylov subspaces,
see, e.g., \cite{Beckermann.11,Beckermann.Kressner.Tobler.13,%
Knizhnerman.Simoncini.11,Druskin.Knizhnerman.Simoncini.11}.

We generalize the bound presented in Proposition~\ref{prop:sym} to the case of the Sylvester equation,
\begin{equation}\label{eq.Sylv}
 AX + X B = b_1 b_2^T,
\end{equation}
with $A$ and $B$ symmetric and positive definite;
without loss of generality, we can assume that $\|b_1\|_\star=\|b_2\|_\star=1$
where $\|\cdot\|_\star$ 
denotes either the Euclidean or the Frobenius norms.

We first recall the Cauchy representation
of the solution matrix $X$ to \eqref{eq.Sylv}. Let us for now only assume that
$A$ and $B$ are positive definite, and not necessarily symmetric. We can write
(see, e.g., \cite{Lancaster1970})
$$
X=\int_0^\infty e^{-t A} b_1 b_2^T e^{-tB} dt.
$$

Consider the approximation 
$X_k=V_kY_kW_k^T$ where $V_k$ and $W_k$ span suitable subspaces and both have orthonormal columns. The 
matrix  $Y_k$ is obtained by imposing the Galerkin condition on $R_k = A X_k + X_k B - b_1b_2^T$, that is
$$
V_k^T R_k W_k = 0 \quad \Leftrightarrow \quad
(V_k^T A V_k) Y_k + Y_k (W_k^T B W_k) - (V_k^Tb_1)(b_2^T W_k) = 0.
$$
Let $A_k := V_k^T A V_k,\;B_k: = W_k^T B W_k$. Thus $Y_k$ is obtained by solving a reduced
Sylvester equation, whose size depends on the approximation space dimensions.
Since the spectrum of $A_k$ ($B_k$) is contained in the spectral region of $A$ ($B$) , 
we have that $\Lambda(A_k)+\Lambda(B_k)\subset\mathbb{C}_+$ and the matrix $Y_k$ can be 
written in integral form as $Y_k=\int_0^\infty e^{-t A_k} (V_k^Tb_1)(b_2^T W_k) e^{-tB_k} dt$ so that
$$
X_k = V_k \int_0^\infty e^{-t A_k} (V_k^Tb_1)(b_2^T W_k) e^{-tB_k} dt\, W_k^T.
$$

Let $x:=e^{-t A} b_1, x_k := V_k e^{-t A_k} (V_k^Tb_1)$,
$y:=e^{-t B} b_2, y_k := W_k e^{-t B_k} (W_k^Tb_2)$. Then, using $\|x\|_\star\leq e^{-t\alpha_{\min}(A)}$ (see, e.g., \cite[Lemma 3.2.1]{Corless2003}), and since $\alpha_{\min}(A_k)\geq\alpha_{\min}(A)$, it holds 
 that $\|x_k\|_\star\leq e^{-t\alpha_{\min}(A)}$. 
Similarly, $\|y\|_\star,\|y_k\|_\star\leq e^{-t\alpha_{\min}(B)}$. Therefore, 
(see also \cite[Lemma 4.7]{Kressner.Tobler.10})
\begin{eqnarray}\label{eq.integral.Sylv}
\|X - X_k\|_\star &= & \left\|\int_0^\infty (xy^T - x_k y_k^T) dt\right\|_\star \notag \\
&= & \frac{1}{2}\left\|\int_0^\infty (x+x_k)(y-y_k)^T + (x-x_k)(y+y_k)^T) dt\right\|_\star \notag\\
 & \leq & \frac{1}{2}\int_0^\infty \Big((\|x\|_\star+\|x_k\|_\star) \|y-y_k\|_\star + 
         \|x-x_k\|_\star (\|y\|_\star+\|y_k\|_\star)\Big) dt\notag \\
 &\leq&  \int_0^\infty \Big(e^{-t\alpha_{\min}(A)}\|y-y_k\|_\star + e^{-t\alpha_{\min}(B)} \|x-x_k\|_\star\Big) dt \notag \\
 & = & \int_0^\infty (\|\hat y-\hat y_k\|_\star +  \|\hat x-\hat x_k\|_\star ) dt ,
\end{eqnarray}
where $\hat y = e^{-t (B+\lambda_{\min}(A) I)} b_2$, 
$\hat x = e^{-t (A+\lambda_{\min}(B) I)} b_1$, and analogously for $\hat y_k, \hat x_k$.
The inequality in~\eqref{eq.integral.Sylv} states that the $\star$-norm of the error 
associated with the Galerkin solution can be bounded by integrating over 
$[0,\infty)$ the errors obtained in the approximation of the exponential of 
the shifted matrices $B+\lambda_{\min}(A) I$ and $A+\lambda_{\min}(B) I$.

In the next proposition we specialize the bound above when the Krylov 
subspaces $\text{range}(V_k)=K_k(A,b_1)$ and $\text{range}(W_k)=K_k(B,b_2)$ are adopted as approximation spaces
and $A$, $B$ are both symmetric and positive definite. To this end, let us define
$\lambda_{\min}(A)$,
 $\lambda_{\max}(A)$, $\lambda_{\min}(B)$, and $\lambda_{\max}(B)$ to be the
extreme eigenvalues of $A$ and $B$, respectively, and
$$
\hat \kappa_A=\frac{\lambda_{\max}(A)+\lambda_{\min}(B)}{\lambda_{\min}(A)+\lambda_{\min}(B)}, \qquad
\hat \kappa_B=\frac{\lambda_{\max}(B)+\lambda_{\min}(A)}{\lambda_{\min}(B)+\lambda_{\min}(A)},
$$
the condition numbers of $A+\lambda_{\min}(B)I$ and $B+\lambda_{\min}(A)I$, respectively.

\begin{proposition}
 Let $A$ and $B$ be spd and 
 $\text{range}(V_k)=K_k(A,b_1)$, $\text{range}(W_k)=K_k(B,b_2)$.
Then the Galerkin approximate solution $X_k=V_kY_kW_k^T$ to \eqref{eq.Sylv} is such that 
 {\small
$$\|X-X_k\|_\star \leq\frac{2}{\lambda_{\min}(A)+\lambda_{\min}(B)}\left(\frac{\sqrt{\hat \kappa_A}+1}{\sqrt{\hat \kappa_A}}\left(\frac{\sqrt{\hat \kappa_A}-1}{\sqrt{\hat \kappa_A}+1}\right)^k+
\frac{\sqrt{\hat \kappa_B}+1}{\sqrt{\hat \kappa_B}}\left(\frac{\sqrt{\hat \kappa_B}-1}{\sqrt{\hat \kappa_B}+1}\right)^k\right),
$$}
where $\|\cdot\|_\star$ denotes either the Euclidean or the Frobenius norm.
\end{proposition}
\begin{proof}
 The proof can be obtained by applying the same arguments of the proof of \cite[Proposition 3.1]{Simoncini.Druskin.09} to the single integrals
  $\int_0^\infty \|\hat y-\hat y_k\|_\star dt$,  $\int_0^\infty \|\hat x-\hat x_k\|_\star dt$ in \eqref{eq.integral.Sylv}.
\end{proof}

Convergence results for generic matrix equations of the form \eqref{eqn:matrixeq} are difficult to derive as no easy-to-handle closed-form solution is known in general. The main difficulty is given by the fact that the exponential of a Kronecker sum $\sum_{j=1}^\ell B_j^T\otimes A_j$ cannot be separated in the product of the exponentials of the single terms
if no further assumptions on $A_j$ and $B_j$ are considered.

By adapting the reasonings proposed in this section, one may be able to deduct error estimates for some special equations of the form
$$\sum_{j,k=0}^\ell \alpha_{j,k}A^jXB^k=F,$$
where the coefficient matrices are given as powers of two \emph{seed} matrices $A$ and $B$, and $\alpha_{j,k}\in\mathbb{R}$ for all $j,k$. Indeed, in this case, the exact solution $X$ can be written in integral form as illustrated in \cite[Theorem 4]{Lancaster1970}. 
However, such derivations deserve a separate analysis.

\section{Comparison with the Kronecker formulation}\label{Comparisons with the Kronecker formulation}
Given a linear matrix equation of the form \eqref{eqn:matrixeq}, the simplest-minded numerical procedure for 
its solution 
consists in applying well-established iterative schemes to the {\em vector} linear system obtained 
from \eqref{eqn:matrixeq} by Kronecker transformations, namely 
\begin{equation}\label{eq.Kron}
\left(\sum_{j=1}^\ell B_j^T\otimes A_j\right)\text{vec}(X)=\text{vec}(F). 
\end{equation}
Sometimes this is the only option as effective algorithms to 
solve \eqref{eqn:matrixeq} in its \emph{natural} matrix equation form are still lacking in the literature in
the most general case.
The methods developed so far require some additional assumptions on the coefficient matrices $A_j$, $B_j$;
see, e.g., \cite{Benner2013a,Shank2016,Jarlebring2018,kressner2015truncated,Powell.Silvester.Simoncini.17}.

In this section we show that exploiting the matrix structure of 
equation~\eqref{eqn:matrixeq} not only leads to numerical algorithms with lower computational costs per iteration and 
modest storage demands, but they also avoid some spectral redundancy encoded in the
problem formulation \eqref{eq.Kron}. Such a redundancy
often leads to a delay in the convergence of the adopted solution scheme when iterative procedures are applied to \eqref{eq.Kron}. 
A similar discussion can be found in \cite[Remark 4.5]{Kressner.Tobler.10} for more general tensor structured problems.

To illustrate this phenomenon we consider a Lyapunov equation of the form \eqref{eqn:Lyap}
with $A\in\mathbb{R}^{n\times n}$ spd and $F=bb^T$, $b\in\mathbb{R}^n$, $\|b\|=1$. We compare the Galerkin method applied to the matrix equation~\eqref{eqn:Lyap} with the CG method  applied to the linear system
\begin{equation}\label{eq.Kron_lyap}
 \mathcal{A}\text{vec}(X)=\text{vec}(bb^T),\quad \mathcal{A}=A\otimes I+I\otimes A\in\mathbb{R}^{n^2\times n^2}.
\end{equation}
Notice that since $A$ is spd, $\mathcal{A}$ is also spd. Let $x = {\rm vec}(X)$ 
be the exact solution to (\ref{eq.Kron_lyap}).
Let the CG initial guess be equal to the zero vector, and let $x_k^{cg}$ be the approximate solution to $x$
obtained after $k$ CG iterations. Then the following classical bound 
for the energy-norm of the error $x - x_k^{cg}$ holds
\begin{equation}\label{eq.CGbound1}
 \frac{\|x-x_k^{cg}\|_\mathcal{A}}{\|x\|_\mathcal{A}}\leq 
2\left(\frac{\sqrt{\kappa}-1}{\sqrt{\kappa}+1}\right)^k, 
\end{equation}
where $\kappa=\lambda_{\max}(\mathcal{A})/\lambda_{\min}(\mathcal{A})=\lambda_{\max}(A)/\lambda_{\min}(A)$. See, e.g., \cite[Theorem 10.2.6]{Golub2013}. This bound may be rather pessimistic since it takes into account neither the role of the right-hand side nor the actual spectral distribution of 
$\mathcal{A}$. See, e.g., \cite{Sluis1986,Beckermann2002,Beckermann2001,Liesen.Strakos.book.12}.

We want to compare the bound in \eqref{eq.CGbound1} with the estimate proposed in Proposition~\ref{prop:sym}, using
the same norms and relative quantities. To this end, we recall that 
for any vector $v$ it holds that
$$
\sqrt{2\lambda_{\min}(A)}\|v\|_2 \leq \|v\|_\mathcal{A} \leq \sqrt{2\lambda_{\max}(A)}\|v\|_2.
$$
In particular, letting $X_k^{cg}\in\mathbb{R}^{n\times n}$ be such that $\text{vec}(X_k^{cg})=x_k^{cg}$, we have
\begin{equation}\label{eq.CGbound2}
 \frac{\|X-X_k^{cg}\|_F}{\|X\|_F} =
 \frac{\|x-x_k^{cg}\|_2}{\|x\|_2}\leq \sqrt{\kappa} \frac{\|x-x_k^{cg}\|_{\cal A}}{\|x\|_{\cal A}}
\leq
  2\sqrt{\kappa}\left(\frac{\sqrt{\kappa}-1}{\sqrt{\kappa}+1}\right)^k. 
\end{equation}
Therefore, to obtain a relative error (in Frobenius norm) of less than $\varepsilon$,
a sufficient number $k_*^{(cg)}$ of CG iterations is given by
$$
k_*^{(cg)} := \frac{ \log\left( \varepsilon/(2\sqrt{\kappa})\right)}%
{ \log\left( (\sqrt{\kappa}-1)/(\sqrt{\kappa}+1)\right)}.
$$

If $X_k$ denotes the approximate solution computed after $k$ iterations of the Galerkin-based
method with $K_k(A,b)$ as approximation space, 
the error norm  bound in \eqref{eqn:bound1} can be written in relative terms as
\begin{equation}\label{eq.Galerkin_bound2}
 \frac{\|X-X_k\|_F}{\|X\|_F}\leq 4(  \sqrt{\hat\kappa}+1 ) \sqrt{\hat\kappa}
\left ( \frac{\sqrt{\hat \kappa} -1}{\sqrt{\hat\kappa}+1}\right )^k ,
\end{equation}
where we used
$\|x\|_2\geq \lambda_{\min}(\mathcal{A}^{-1})\,\|\text{vec}(bb^T)\|_F= 1/(\lambda_{\max}+\lambda_{\min})$. 
Once again, to obtain a relative error (in Frobenius norm) of less than $\varepsilon$,
a sufficient number $k_*^{(G)}$ of iterations is
given by
$$
k_*^{(G)} := \frac{ \log\left( \varepsilon/(4\sqrt{\hat\kappa} (\sqrt{\hat\kappa}+1))\right)}%
{ \log\left( (\sqrt{\hat\kappa}-1)/(\sqrt{\hat\kappa}+1)\right)}.
$$

The bounds \eqref{eq.CGbound2}--\eqref{eq.Galerkin_bound2} show that
 the asymptotic behavior of the relative error norms of CG and the Galerkin method are guided by
$\kappa$ and $ \hat \kappa$, respectively, where
%
$\hat \kappa$ is always smaller than $\kappa$, for $\kappa>1$. Indeed,
$$
\hat \kappa=
\frac{\lambda_{\max}(A)+\lambda_{\min}(A)}{2 \lambda_{\min}(A)}=\frac{1}{2}\kappa+\frac{1}{2} .
$$

The worse conditioning of the linear system formulation \eqref{eq.Kron_lyap} 
may lead to a delay in the convergence of CG so that, for a fixed threshold, CG may require more iterations to converge than 
 the Galerkin method applied to the matrix equation \eqref{eqn:Lyap}. 
This is numerically illustrated in the examples below.

We once again stress that the similarities of the two formulations (matrix equation and Kronecker form) highlight the
fact that what makes the matrix equation context more efficient than CG on ${\cal A}x =b$
is the special choice of the approximation space, that is ${\cal K}_m = \text{range}(V_k \otimes V_k)$, which 
heavily takes into account the Kronecker sum structure of ${\cal A}$. On the other hand, CG applied blindly on ${\cal A}$
generates a redundant approximation space.

\begin{num_example}\label{Ex.1}
{\rm
We consider the spd matrix $A=QD Q^T\in\mathbb{R}^{n\times n}$, where $D$ is a diagonal matrix whose diagonal entries are
uniformly distributed (in logarithmic scale) values between 1 and 100,
and $Q$ is orthogonal. This means that $\kappa=100$ and $\hat \kappa=50.5$ for any $n$. 
The vector $b\in\mathbb{R}^n$ is a random vector with unit norm.
%

For $\varepsilon=10^{-6}$, a direct computation shows 
that $k_*^{(G)}=68$ iterations of the Galerkin method are sufficient to get 
$\|X-X_{k_*^{(G)}}\|_F/\|X\|_F\leq \varepsilon$, whereas according to the bound~\eqref{eq.CGbound2}, $k_*^{(cg)}=84$ iterations
 are required for CG to reach the same accuracy when solving ${\cal A} x = b$. In practice, the number of
actual iterations can be lower, since this estimate is obtained from a bound.


Figure~\ref{Ex.1_Fig.1} reports the error convergence history
of the two iterations, using logarithmic scale for $n=1000$. 
 The two methods are stopped as soon as the relative error norm becomes smaller than $\varepsilon$.
The ``exact'' solution $X$ was computed
with the Bartels-Stewart method \cite{Bartels1972}, which was feasible due to the small problem size.

\begin{figure}
  \centering    
  \caption{Example~\ref{Ex.1}. Relative error norms produced by the Galerkin and CG methods.  \label{Ex.1_Fig.1}}
	\begin{tikzpicture}
  \begin{semilogyaxis}[width=0.8\linewidth, height=.27\textheight,
    legend pos = north east, 
    xlabel = $k$, ylabel = Relative Error ($F$-norm)]
    \addplot+[ mark = o, blue] table[x index=0, y index=1]  {err_cg2.dat};
     \addplot+ [ mark = square, red]table[x index=0, y index=1]  {err_galerkin2.dat};
      \legend{CG (lin. system),Galerkin (matrix eq.)};
  \end{semilogyaxis}          
\end{tikzpicture}
\end{figure}

Both methods require slightly fewer iterations than predicted by the bounds. Nonetheless,
we can still appreciate that CG applied to the linear system~\eqref{eq.Kron_lyap} requires 
more iterations than the Galerkin method applied to the matrix equation~\eqref{eqn:Lyap} to achieve the 
same prescribed accuracy.
 }
\end{num_example}

\begin{num_example}\label{Ex.4}
{\rm
We modify the data of Example \ref{Ex.1} by replacing $\lambda_{\min}(A)=1$ with
$\lambda_{\min}(A)=\lambda_1=0.001$, while the other eigenvalues are such that 
$\lambda_2=2,\ldots,\lambda_n=n$.
Here $b\in\mathbb{R}^n$ is the vector of all ones normalized.
The relative error energy norm obtained by CG and the Galerkin method is reported in the
left plot of Figure~\ref{Ex.4_Fig.1} for $n=100$. Notice that with such a $n$ we have
$\kappa = 10^5, \hat\kappa \approx 5\cdot 10^4$.

\begin{figure}\label{fig:lapl2D_rpk}
	\begin{center}
	\begin{tabular}{cc}
		{	\begin{tikzpicture}
  \begin{semilogyaxis}[width=0.46\linewidth, height=.38\textheight,
    legend pos = south west,
    xlabel = $k$, ylabel = Relative Error ($\mathcal{S}$-norm)]
    \addplot+[ mark = square, red] table[x index=0, y index=1]  {err_galerkin_referee_example.dat};
    \addplot+[ mark = o, blue] table[x index=0, y index=1]  {err_cg_referee_example.dat};
      \legend{Galerkin,CG };
  \end{semilogyaxis}          
\end{tikzpicture}
}
		&
		{	\begin{tikzpicture}
  \begin{semilogyaxis}
  [width=0.46\linewidth, height=.38\textheight,
    legend pos = north east, legend style={at={(0.6,0.70)},anchor=west},
    xlabel = $k$, ylabel = Ritz values]
    \addplot+[no marks, dashed, black] table[x index=0, y index=5]  {RitzValues.dat};
    \addplot+[no marks, dashed, black] table[x index=0, y index=6]  {RitzValues.dat};
    \addplot+[no marks, dashed, black] table[x index=0, y index=7]  {RitzValues.dat};
    \addplot+[no marks, dashed, black] table[x index=0, y index=8]  {RitzValues.dat};
    
    \addplot+[ mark = square, only marks, red] table[x index=0, y index=1]  {RitzValues.dat};
    \addplot+[ mark = square, only marks, red] table[x index=0, y index=2]  {RitzValues.dat};
    \addplot+[ mark = o, only marks, blue] table[x index=0, y index=3]  {RitzValues.dat};
    \addplot+[ mark = o, only marks, blue] table[x index=0, y index=4]  {RitzValues.dat};
      \legend{,,,,,Galerkin,CG };
    
  \end{semilogyaxis}          
\end{tikzpicture}
}
	\end{tabular}
	\end{center}
  \caption{Example~\ref{Ex.4}. Left: Relative error energy norms produced by the 
Galerkin and CG methods. Right: Maximum and minimum Ritz values computed at the $k$-th iteration of CG and the Galerkin method. The dashed black line represents the quantity to be approximated, namely $\lambda_{\min}(A)$, $\lambda_{\max}(A)$, $\lambda_{\min}(\mathcal{A})$, and $\lambda_{\max}(\mathcal{A})$.  \label{Ex.4_Fig.1}}
\end{figure}

Both methods stagnate in the initial phase of the solution process,
followed by a rapid convergence afterwards. The stagnation phase is significantly longer for CG,
contributing to the overall CG delay.
A closer look at the convergence history of the Ritz values towards the eigenvalues
of the corresponding coefficient matrices provides a better understanding.
In this setting, the Ritz values for the Galerkin
and CG methods are the eigenvalues of $V_k^TAV_k$ and of $Q_k^T {\cal A}Q_k$, respectively,
where the columns of $Q_k$ are the orthonormal basis of the space generated by CG (here we used the Arnoldi
procedure to compute $Q_k$).
%
Recalling that 
$\lambda_{\min}(A)=0.001$, $\lambda_{\max}(A)=100$ so that $\lambda_{\min}(\mathcal{A})=0.002$, $\lambda_{\max}(\mathcal{A})=200$, 
the right plot of Figure~\ref{Ex.4_Fig.1} reports the convergence history of the extreme
Ritz values computed at the $k$-th iteration of both CG and the Galerkin method for $k=1,\ldots,82$ (the dashed lines
indicate the target eigenvalues).
The Ritz value tending to the largest eigenvalue converges in very few iterations. 
For each approach, the Ritz value approximating
the smallest eigenvalue takes many more iterations to converge, and these 
iterations seem to match the stagnation phase observed in the left plot.
It appears that the matrix Galerkin approximation space is able to implicitly capture the
Kronecker structure of the eigenvector associated with $\lambda_{\min}({\cal A})$ much earlier than what CG can
do by using the unstructured basis $Q_k$. Once again, this emphasizes
the importance of the Kronecker basis determined by the matrix Galerkin method.

 }
\end{num_example}


\section{Petrov-Galerkin method and residual minimization}\label{Petrov-Galerkin methods and residual minimization}

Whenever the linear operator $\mathcal{S}$ is not spd, the Galerkin method does not necessarily lead to a minimization of the error norm. 
As in the linear system setting, a numerical procedure fulfilling an optimality condition can be obtained by imposing a Petrov-Galerkin condition on the residual also when solving linear matrix equations. For the case of Lyapunov and Sylvester equations, this strategy has been already explored in, e.g., \cite{Lin2013,Hu1992}, and in this section we are going to present some considerations about the application of Petrov-Galerkin methods to the solution of generic linear matrix equations of the form \eqref{eqn:matrixeq}.

We first recall the Petrov-Galerkin framework applied to the solution of the linear system~\eqref{eqn:mainkron}. If the columns of $\mathcal{V}_m\in\mathbb{R}^{N\times m}$ constitute an orthonormal basis for the selected trial space $\mathcal{K}_m$, we want to compute a solution $x_m=\mathcal{V}_my_m$,
where $y_m\in\mathbb{R}^m$ is calculated by imposing a Petrov-Galerkin condition on the residual vector $r_m=f-\mathcal{M}x_m$. In its full generality, such a condition reads
\begin{equation}\label{PetrovGalerkin1}
r_m\,\bot\,\mathcal{L}_m,\;\text{i.e.,}\; \mathcal{W}_m^Tr_m=0,\;\text{range}(\mathcal{W}_m)=\mathcal{L}_m, 
\end{equation}
where $\mathcal{L}_m$ is the chosen test space. See, e.g., \cite[Chapter 5]{Saad2003}. 

For the particular choice $\mathcal{L}_m=\mathcal{M}\mathcal{K}_m$, the 
condition in \eqref{PetrovGalerkin1} is equivalent to computing $x_m$ as the 
minimizer of the residual norm over $\mathcal{K}_m$, namely
$$x_m=\argmin_{x\in\mathcal{K}_m}\|f-\mathcal{M}x\|.$$
See, e.g., \cite[Proposition 5.3]{Saad2003}.
With the selection $\mathcal{K}_m=K_m(\mathcal{M},f)$, the minimization problem above can be significantly simplified by 
exploiting the Arnoldi relation; this is the foundation of some of the most popular minimal residual methods 
for linear systems such as, e.g., MINRES \cite{Paige1975} and GMRES \cite{Schultz1986}.

A similar approach can be pursued for the solution of linear matrix equations.
Indeed, let $N=n p$ and
consider $V_k\in\mathbb{R}^{n\times k}$, $W_k\in\mathbb{R}^{p\times k}$ with full column rank \footnote{Once again, the two matrices
may have different column dimensions, that is $V_{k_1}\in\mathbb{R}^{n\times k_1}$, $W_{k_2}\in\mathbb{R}^{p\times k_2}$. For
the sake of clarity in the exposition, we limit our presentation to the case $k_1=k=k_2$.}, 
and let
$\text{range}(V_k)$, $\text{range}(W_k)$, be the corresponding left and right approximation spaces. With 
$\mathcal{S}_\ell$ as in Definition~\ref{def:S}, 
we can formally set $\mathcal{K}_m=\text{range}(W_k\otimes V_k)$ and 
$\mathcal{L}_m=\mathcal{S}_\ell \mathcal{K}_m$.
An approximate solution in the form $X_k=V_kY_kW_k^T$, with $Y_k\in\mathbb{R}^{k\times k}$, can be
determined by imposing the condition
\eqref{PetrovGalerkin1} to the vector form of the residual matrix $R_k=\mathcal{S} (V_kY_kW_k^T)-F$.

Petrov-Galerkin methods for \eqref{eqn:matrixeq}  thus seek a solution $X_k=V_kY_kW_k^T$ by solving
$$\min_{x\in\text{range}(W_k\otimes V_k)}\|\text{vec}(F)-\mathcal{S}_\ell x\|_2= 
\min_{y\in\mathbb{R}^{k^2}}\|\text{vec}(F)-\mathcal{S}_\ell (W_k\otimes V_k)y\|_2,
$$
that is
\begin{equation}\label{PetrovGalerkin_matrixeq}
\min_{X=V_kYW_k^T}\|F-\mathcal{S}(X)\|_F=\min_{Y\in\mathbb{R}^{k\times k}}\|F-\mathcal{S}(V_kYW_k^T)\|_F. 
\end{equation}

In spite of their appealing minimization property, minimal residual methods are not very popular in the 
matrix equation literature. This is mainly due to the difficulty in dealing with the numerical solution of the
minimization problem~\eqref{PetrovGalerkin_matrixeq}. In general, one can apply an operator-oriented 
(preconditioned) CG method to the normal equations as 
\begin{equation}\label{normal_eq}
Y_k=\argmin_{Y\in\mathbb{R}^{k\times k}}\|F-\mathcal{S}(V_kYW_k^T)\|_F\qquad
\Leftrightarrow\qquad \mathcal{S}^*(F-\mathcal{S}(V_kY_kW_k^T))=0,
\end{equation}
where $\mathcal{S}^*$ is the adjoint of $\mathcal{S}$, namely 
$$\begin{array}{lrll}
  {\cal S}^* :& \mathbb{R}^{n\times p}&\rightarrow&\mathbb{R}^{n\times p}\\
  & X &\mapsto& \displaystyle\sum_{j=1}^\ell A_j^T X B_j^T.\\
  \end{array}
$$
If $\text{range}(V_k)$ and $\text{range}(W_k)$ are general spaces, the solution 
of \eqref{normal_eq} can be very expensive in terms of both computational time and memory requirements. 

In \cite{Lin2013}, the authors consider \eqref{normal_eq} in the case of the Lyapunov equation \eqref{eqn:Lyap} with $F$ low-rank and negative semidefinite.
In particular, if $F=-bb^T,$ $b\in\mathbb{R}^{n\times q}$, $q\ll n$, they employ
the approximation spaces
$\text{range}(V_k)=\text{range}(W_k)$ such that 
$b=V_1L_b$ for some $L_b\in\mathbb{R}^{q\times q}$, $q=\text{rank}(C)$, and satisfying an Arnoldi-like relation of the form
$$AV_k=[V_k,\breve{V}_{k+1}]\underline{H}_k,$$
for $[V_k,\breve{V}_{k+1}]\in\mathbb{R}^{n\times (k+1)q}$ having orthonormal columns and $\underline{H}_k\in\mathbb{R}
^{(k+1)q\times kq}$.
In this case, the minimization problem \eqref{PetrovGalerkin_matrixeq} can be written as 
\begin{equation}\label{minimal_residual}
 Y_k=\argmin_{Y\in\mathbb{R}^{kq\times kq}}\left\|\underline{H}_kY[I_{kq},0]+\begin{bmatrix}
                 I_{kq} \\                                                  
                  0\\                                                 \end{bmatrix}Y\underline{H}_k^T+\begin{bmatrix}
                  L_bL_b^T & 0 \\
                  0 & 0 \\
                  \end{bmatrix}
\right\|_F,
\end{equation}
and three different methods for its solution are illustrated. 

If the coefficient matrix $A$ in \eqref{eqn:Lyap} is stable (antistable) and $F$ is symmetric negative semidefinite, the exact solution $X$ is symmetric positive (negative) semidefinite. See, e.g., \cite{Snyders1970}.
However, as reported in \cite{Lin2013}, the numerical solution $X_k=V_kY_kV_k^T$ is not guaranteed to be semidefinite if $Y_k$ is computed as in \eqref{minimal_residual}.

In \cite[Section 3.4]{Lin2013} it is shown that \eqref{minimal_residual}
 is equivalent to computing $Y_k$ as the solution of the \emph{generalized} Sylvester equation
 \begin{equation}\label{gen_Lyap}
\underline{H}_k^T\underline{H}_kY+Y\underline{H}_k^T\underline{H}_k+H_kYH_k+H_k^TYH_K^T+D=0,  
 \end{equation}
 where 
$$
D:=\underline{H}_k^T
\begin{bmatrix}
                  L_bL_b^T & 0 \\
                  0 & 0 \\
                  \end{bmatrix}\begin{bmatrix}
                 I_{kq} \\                                                  
                  0\\                                                 
\end{bmatrix}
 +[I_{kq},0]\begin{bmatrix}
                  L_bL_b^T & 0 \\
                  0 & 0 \\
                  \end{bmatrix}
 \underline{H}_k= H_k^T\begin{bmatrix}
                  L_bL_b^T & 0 \\
                  0 & 0 \\
                  \end{bmatrix}+\begin{bmatrix}
                  L_bL_b^T & 0 \\
                  0 & 0 \\
                  \end{bmatrix}H_k ,
$$
so that $D$ is symmetric but indefinite.
 This is one of the main obstacles in proving the semidefiniteness of $Y_k$ through the matrix formulation \eqref{gen_Lyap}. 
Without further hypotheses, the symmetric matrix $Y_k$ solving \eqref{gen_Lyap} is indefinite in general,
thus preventing $Y_k$ from preserving the semidefiniteness property of the solution
to be approximated.

From a computational viewpoint, if resorting to a Kronecker form is excluded, 
the generalized Sylvester equation \eqref{gen_Lyap} can be solved by means of the methods described in \cite{Lin2013} and its references. In addition,
setting
$\mathfrak{L}(Z)=\underline{H}_k^T\underline{H}_kZ+Z\underline{H}_k^T\underline{H}_k$ and 
$\mathfrak{N}(Z) =H_kZH_k+H_k^TZH_k^T$, fixed point iterations can be used
whenever the spectral radius of the
operator $\mathfrak{L}^{-1}(\mathfrak{N}(\cdot))$ is less than one; 
see, e.g., \cite{Damm.08,Shank2016,Jarlebring2018} for various implementations.


\subsection{A constrained residual minimization approach for Lyapunov equations}
To cope with the lack of semidefiniteness in the least squares problem approach,
we propose to explicitly impose the semidefiniteness as a constraint. For instance, if a negative semidefinite solution is sought, the problem
becomes
\begin{equation}\label{eqn:constr_lsqr}
 Y_k=\argmin_{Y\in\mathbb{R}^{kq\times kq}\atop Y \leq 0}
\left\|\underline{H}_kY[I_{kq},0]+
\begin{bmatrix}
I_{kq} \\                                                  
0\\                                                 
\end{bmatrix}
Y\underline{H}_k^T+
\begin{bmatrix}
L_bL_b^T & 0 \\
0 & 0 \\
\end{bmatrix}
\right\|_F.
\end{equation}
To numerically solve this inequality constrained least squares problem, we
consider a linear matrix inequalities (LMI) approach, which suits very well
the matrix equation framework \cite{BEFB94,Skeltonetal.98}; other general purpose methods  could also
be considered \cite{Anjos2012,Malick2004}.

%
%

In the LMI context, (\ref{eqn:constr_lsqr}) can be stated as the following
semidefiniteness matrix inequalities
$$
Y \leq 0, \qquad 
\begin{bmatrix}
I & {\rm vec}(\underline{H}_k Y J^T + J Y \underline{H}_k^T + M) \\
{\rm vec}(\underline{H}_k Y J^T + J Y \underline{H}_k^T + M)^T & \gamma 
\end{bmatrix} 
\ge 0,
$$
for the unknown matrix $Y$ and scalar $\gamma >0$; here $J=[I_{kq};0]$ and $M = [L_bL_b^T,0;0, 0]$.

\begin{num_example}\label{Ex.2}
{\rm
We consider the Lyapunov equation \eqref{eqn:Lyap}  with $A=QD Q^{-1}\in\mathbb{R}^{n\times n}$, $D$ as in Example~\ref{Ex.1}, $Q$ a random matrix, and $F=-bb^T$, where  
$b\in\mathbb{R}^n$ is a random vector with unit norm.

Since $A$ is antistable and the right-hand side is symmetric negative semidefinite, the solution $X$ is symmetric negative semidefinite and we thus expect the approximate solution $X_k=V_kY_kV_k^T$ to be so as well.

We apply the Petrov-Galerkin method discussed in this section in the solution process and we adopt the Krylov subspace as approximation space, i.e., $\text{range}(V_k)=K_k(A,b)$. The matrix $Y_k$ is computed in two different ways. We first solve the unconstrained minimization problem \eqref{minimal_residual} getting the matrix  $Y_k^{\text{uncon}}$. In particular, $Y_k^{\text{uncon}}$ is computed by applying a (preconditioned) CG method to the matrix equation~\eqref{gen_Lyap}. See, e.g., \cite{Lin2013}. Then, we compute $Y_k^{\text{const}}$ by solving the constrained minimization problem \eqref{eqn:constr_lsqr}. 
The Petrov-Galerkin method is stopped as soon as the relative residual norm becomes smaller than $10^{-6}$.

%
%
%
%

\begin{figure}
\centering 
\caption{Example~\ref{Ex.2}. Intervals  $[\min_j\{\lambda_j(Y_k^{\text{uncon}})\geq 0\},\max_j\{\lambda_j(Y_k^{\text{uncon}})\geq 0\}]$
for all $k=1,\ldots,68$. $n=1000$. 
  } \label{fig:3}
\begin{tikzpicture}
    \begin{semilogyaxis}[  legend pos = south west,xmin=1,xmax=70,xscale=2, xlabel = Iterations, nodes near coords={
    \rotatebox{0}{%
    \small
    \pgfmathprintnumber[fixed,precision=0,zerofill]\pgfplotspointmeta
    }}]
      \addplot+ [scatter,
        only marks,error bar legend 1,
        scatter src=explicit symbolic,
        scatter/classes={
          A0={mark=none,red,thick
          }},
        error bars/.cd,
        y dir=both,
        y explicit, 
        error bar style={color=red,dotted,mark=square*}]
      table[x=x,y=y,y error=err,meta=class] {eig0_10.dat};
    \addplot+ [scatter,error bar legend 2,
        only marks,
        scatter src=explicit symbolic,
        scatter/classes={
          A0={mark=none,blue,
          }},
        error bars/.cd,
        y dir=both,
        y explicit, 
        error bar style={color=blue}]
      table[x=x,y=y,y error=err,meta=class] {eig11_20.dat};
      \addplot [scatter,error bar legend 3,
        only marks,
        scatter src=explicit symbolic,
        scatter/classes={
          A0={mark=none,black,
          }},
        error bars/.cd,
        y dir=both,
        y explicit, 
        error bar style={color=black,dashed}]
      table[x=x,y=y,y error=err,meta=class] {eig21_30.dat};
      \legend{\#$\{\lambda_j\geq 0\}\in [0\text{,}10]$,\#$\{\lambda_j\geq 0\}\in [11\text{,}20]$,\#$\{\lambda_j\geq 0\}\in [21\text{,}26]$};  
    \end{semilogyaxis}
\end{tikzpicture}
\end{figure}
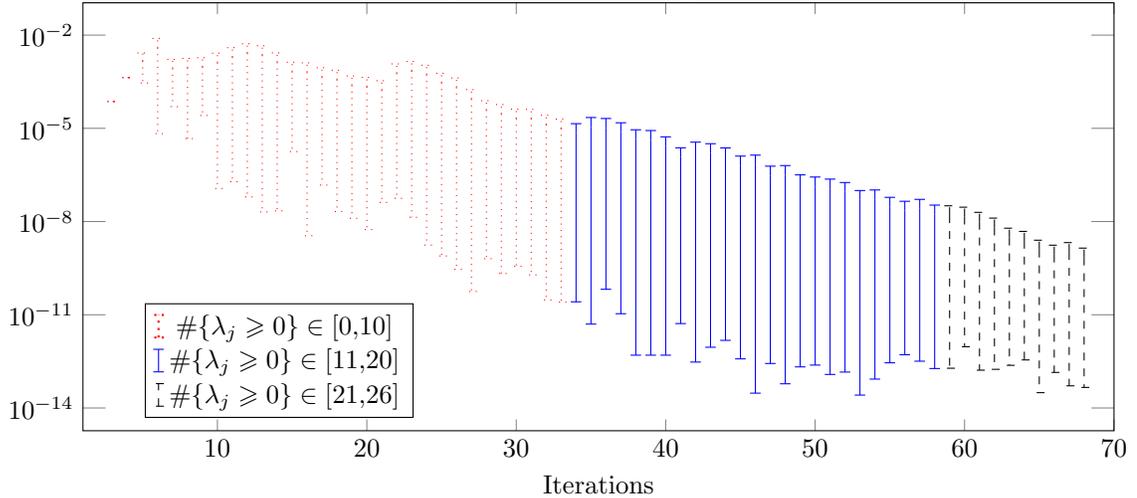

In Figure~\ref{fig:3} we plot the intervals  
$[\min_j\{\lambda_j(Y_k^{\text{uncon}})\geq 0\},\max_j\{\lambda_j(Y_k^{\text{uncon}})\geq 0\}]$
of the undesired positive eigenvalues of $Y_k^{\text{uncon}}$
for all $k$ for the case $n=1000$.
For $k=1,2$, $Y_k^{\text{uncon}}$ has all negative eigenvalues,
while it starts being indefinite for $k\geq3$ so that $X_k^{\text{uncon}}=V_kY_k^{\text{uncon}}V_k^T$ 
is an indefinite approximation to the negative semidefinite $X$. Nonetheless,
for $k=68$, the (undesired) positive eigenvalues of $X_k^{\text{uncon}}$
are small enough so as to still allow a sufficiently accurate approximation,
in terms of relative residual norm. 
On the other hand, this problem is not encountered with $Y_k^{\text{constr}}$,
thanks to the explicit negative semidefiniteness constraint in the formulation \eqref{eqn:constr_lsqr}.

From the legend of Figure~\ref{fig:3}, we can see that the number of positive eigenvalues of $Y_k^{\text{uncon}}$ increases as the iteration proceed, even though they diminish in magnitude. The latter trend is not surprising. Indeed, even if $Y_k^{\text{uncon}}$ is computed by \eqref{minimal_residual}, the Petrov-Galerkin method is converging 
towards the negative semidefinite solution $X$ and, for an approximation space spanning the whole
$\RR^n$, the method would retrieve the exact solution, regardless of the minimization problem \eqref{minimal_residual}.

 
 We would like to point out that both tested variants of the Petrov-Galerkin method 
needed 68 iterations to converge and the actual values of the residual norm provided 
by \eqref{minimal_residual} and \eqref{eqn:constr_lsqr} were always very similar to each other, 
during the whole convergence history. This phenomenon surely deserves further studies as, 
in principle, \eqref{eqn:constr_lsqr} leads to a residual norm that is greater or 
equal than the one provided by \eqref{minimal_residual}, while the two solutions (constrained
and unconstrained) do not necessarily have to be close to each other.
 
 }
\end{num_example}

In our computational experiments, we have used the Yalmip software 
\cite{Lofberg2004} running the algorithm Sedumi in Matlab \cite{Sedumi}.
This algorithm is rather expensive and computing the solution $Y_k$ to 
\eqref{eqn:constr_lsqr} at each Krylov iteration $k$ often leads to a very time 
consuming solution procedure. We think this issue can be fixed in different ways. For instance, 
one may compute $Y_k$, and thus check the residual norm, only periodically, say every 
$d\geq1$ iterations. Moreover, the explicit solution $Y_k$ is required only at convergence 
while we just need the value of the residual norm during the Krylov routine. It may be possible to compute such a residual norm without calculating the whole $Y_k$ as it is done in \cite{Palitta2018} for the Galerkin method and in \cite{Lin2013} for the Petrov-Galerkin technique equipped with the unconstrained minimization problem \eqref{minimal_residual}.

The study of the aforementioned enhancements and, more in general, the employment 
of constrained minimization procedures in the solution of linear matrix equations 
will be the topic of future research.


\section{Conclusions}\label{Conclusions}
We have shown that the optimality properties of Galerkin 
and Petrov-Galerkin methods naturally extend to the general linear matrix equation setting.
Such features do not depend on the adopted approximation spaces even though, in actual computations,
fast convergence depends on the suitable subspace selection.
Identifying effective subspaces for general (multiterm) linear matrix equations depends on the problem
at hand, and it may seem easier to recast the solution in terms of a large 
vector linear system. 
On the other hand, the vector form can be extremely memory consuming, while
 the vector linear system encodes some spectral redundancy which may cause a delay in the converge of the adopted iterative solution scheme.

Petrov-Galerkin schemes require to solve a matrix minimization problem at each iteration and we have suggested to explicitly incorporate a semidefiniteness constraint in its formulation. To the best of our knowledge, such approach has never been proposed in the literature and the employment of constrained optimization techniques in the context of Petrov-Galerkin methods for linear matrix equations opens many new research directions.

\section*{Acknowledgements}
 Both authors are members of the Italian INdAM Research group GNCS.
 
 We thank the two anonymous reviewers for their insightful remarks. 
\bibliography{galerkin}


\end{document}

%% file: generic.tex
\usepackage{amsmath,amssymb}
\usepackage{cite}
\usepackage{mathtools}
\usepackage{tikz}
\usepackage{pgfplots}
\usepackage{pgfplotstable}
\usepackage{geometry}
\usepackage{color}
\usepackage{booktabs}
\usepackage{framed}
\pgfplotsset{compat=1.9}

\pgfplotstableset{
	every head row/.style={before row=\toprule,after row=\midrule},
	clear infinite
}

\usepackage{amsopn}

\renewcommand{\leq}{\leqslant}
\renewcommand{\geq}{\geqslant}